\documentclass[11pt]{article}
\usepackage{amsmath, amsfonts,amssymb, amsthm}
\usepackage{tikz}
\usetikzlibrary{patterns}
\usepackage{enumitem}
\usepackage{graphicx}
\usepackage[a4paper, top=1in, bottom=1.15in, left=1.15in, right=1.2in]{geometry}

\newtheorem{thm}{Theorem}[section]
\newtheorem{exa}[thm]{Example}
\newtheorem{cla}[thm]{Claim}
\newtheorem{prop}[thm]{Proposition}

\newtheorem{lem}[thm]{Lemma}

\newtheorem{fact}[thm]{Fact}

\newcommand{\floor}[1]{\left\lfloor#1\right\rfloor}
\newcommand{\ceiling}[1]{\left\lceil#1\right\rceil}
\newcommand{\dk}{\ceiling{\frac{n}{2}}+\floor{\frac{n+2}{2\ceiling{\frac{k+1}{2}}}}-\frac{n}{k}}
\newcommand{\Dk}{\begin{cases} 
\ceiling{\frac{n}{2}}+\floor{\frac{n+2}{k+1}}-\frac{n}{k} & : k \text{ odd }\\
\frac{n}{2}+\floor{\frac{n+2}{k+2}}-\frac{n}{k} & : k \text{ even }  
\end{cases}}

\newcommand{\ck}{\ceiling{\frac{k+1}{2}}}
\newcommand{\tck}{2\ceiling{\frac{k+1}{2}}}

\newcommand{\tbf}[1]{\textbf{#1}}

\begin{document}
\title{On Hamiltonian cycles in balanced $k$-partite graphs}
\author{Louis DeBiasio\thanks{Department of Mathematics; Miami University; Oxford, OH.  {\tt debiasld@miamioh.edu}, {\tt spanienm@miamioh.edu}} \thanks{Research supported in part by Simons Foundation Collaboration Grant \# 283194.} and Nicholas Spanier\footnotemark[1]}

\maketitle
\begin{abstract}
For all integers $k$ with $k\geq 2$, if $G$ is a balanced $k$-partite graph on $n\geq 3$ vertices with minimum
degree at least
\[
\ceiling{\frac{n}{2}}+\floor{\frac{n+2}{2\ceiling{\frac{k+1}{2}}}}-\frac{n}{k}=\begin{cases} 
\ceiling{\frac{n}{2}}+\floor{\frac{n+2}{k+1}}-\frac{n}{k} & : k \text{ odd }\\
\frac{n}{2}+\floor{\frac{n+2}{k+2}}-\frac{n}{k} & : k \text{ even }  
\end{cases},
\]
then $G$ has a Hamiltonian cycle unless $k=2$ and 4 divides $n$, or $k=\frac{n}{2}$ and 4 divides $n$.  In the case where $k=2$ and 4 divides $n$, or $k=\frac{n}{2}$ and 4 divides $n$, we can characterize the graphs which do not have a Hamiltonian cycle and see that $\ceiling{\frac{n}{2}}+\floor{\frac{n+2}{2\ceiling{\frac{k+1}{2}}}}-\frac{n}{k}+1$ suffices.  This result is tight for all $k\geq 2$ and $n\geq 3$ divisible by $k$. 
\end{abstract}

\section{Introduction}

The study of Hamiltonian cycles in balanced $k$-partite graphs begins with the following classic results of Dirac, and Moon and Moser.  Dirac \cite{D} proved that for all graphs $G$ on $n\geq 3$ vertices, if $\delta(G)\geq \ceiling{\frac{n}{2}}$, then $G$ has a Hamiltonian cycle. Moon and Moser \cite{MM} proved that for all balanced bipartite graphs $G$ on $n\geq 4$ vertices, if $\delta(G)\geq \frac{n+2}{4}$, then $G$ has a Hamiltonian cycle.

Over 30 years later Chen, Faudree, Gould, Jacobson, and Lesniak \cite{CFGJL} beautifully tied these results together by proving that for all $k\geq 2$, if $G$ is a balanced $k$-partite graph on $n$ vertices with 
\begin{equation}\label{cfgjl}
\delta(G)> \frac{n}{2}+\frac{n}{2\ceiling{\frac{k+1}{2}}}-\frac{n}{k},
\end{equation}
then $G$ has a Hamiltonian cycle.  It turns out that while their result is nearly optimal, in most cases the degree condition can be improved by 1.  The purpose of this note is simply to provide the precise minimum degree condition in all cases thereby filling the lacuna in the above result (and as we point out in the Appendix, it is unfortunately not as simple as replacing the strict inequality in \eqref{cfgjl} with a weak inequality).  

\begin{thm}\label{main} Let $k$ be an integer with $k\geq 2$.  For all balanced $k$-partite graphs $G$ on $n$ vertices, if
\begin{equation}\label{cf}
\delta (G) \geq \dk=\Dk,
\end{equation}
then $G$ has a Hamiltonian cycle unless $k=2$ and 4 divides $n$, or $k=\frac{n}{2}$ and 4 divides $n$.  
\end{thm}

Since a graph on $n$ vertices can be viewed as a $k$-partite graph with $k=n$, note that when $k=n$, we have $\dk=\ceiling{\frac{n}{2}}$ and thus Theorem \ref{main} reduces to Dirac's theorem.  When $k=2$, we have $\dk=\floor{\frac{n+2}{4}}$ and thus when 4 does not divide $n$, Theorem \ref{main} reduces to Moon and Moser's theorem; and when 4 does divide $n$, Ferrara, Jacobson, and Powell \cite{FJP} characterized all balanced bipartite graphs $G$ on $n\geq 4$ vertices such that $\delta(G)\geq \frac{n}{4}$, yet $G$ does not have a Hamiltonian cycle.  So our proof will only handle the cases when $3\leq k\leq \frac{n}{2}$.

We will also prove the following which will handle the case when $k=\frac{n}{2}$ and 4 divides $n$.  Together with the results in \cite{FJP}, this gives a complete characterization of balanced $k$-partite graphs $G$ on $n$ vertices which satisfy $\delta (G) \geq \dk$, but do not have a Hamiltonian cycle.

\begin{prop}\label{prop:n=2k}
Let $n\geq 8$ be divisible by $4$, let $k=\frac{n}{2}$, and let $G$ be a balanced $k$-partite graph on $n$ vertices.  If $\delta(G)\geq \frac{n}{2}+\floor{\frac{n+2}{k+2}}-\frac{n}{k}=\frac{n}{2}-1$ and $G$ does not have a Hamiltonian cycle, then $G$ belongs to one of the families of examples described in Example \ref{n=2k}.
\end{prop}

\subsection{Overview}

We give the lower bound examples in Section \ref{sec:example}, we collect the main lemmas in Section \ref{sec:example} (while it is a combination of existing results, Lemma \ref{domcycle} may be of independent interest), we deal with the first two exceptions in Section \ref{sec:filter} before starting the main proof in Section \ref{sec:main}. Finally, in the Appendix, we collect some numerical lemmas which are needed because of the floors and ceilings appearing in \eqref{cf}.  

This project grew out of an earlier work of the first author together with Krueger, Pritikin, and Thompson \cite{DKPT}, where we considered Hamiltonian cycles in unbalanced $k$-partite graphs for $k\geq 3$.  The upcoming Example \ref{tight_gen} first appeared in a more general form in \cite{DKPT}.  In fact, it was this example which indicated to us that \eqref{cfgjl} is not always tight.  By using Theorem \ref{chv} and Lemma \ref{domcycle}, we were able to streamline the original proof of Chen et al.\ with the correct degree condition; however, because of the unexpected (to us) exceptional cases which arose when $k=\frac{n}{2}$, our overall proof didn't end up being any shorter than the original.  Again, we emphasize that Chen et al.\ have a beautiful result which places Dirac's theorem and Moon and Moser's theorem on a common spectrum.  It is only because of the fundamental nature of these results that we have expended the effort necessary to provide the tight degree condition in all cases.

\subsection{Notation}

For $S\subseteq V(G)$, we let $N(S) = \bigcup_{v}N(v)$ and $\overline{S}= V(G)\setminus S$.
Given disjoint sets $A, B\subseteq V(G)$, we let $\delta(A,B)=\min\{|N(v)\cap B|: v\in A\}|$.

Given a cycle $v_1v_2\dots v_kv_1$, $i\in [k]$, and an integer $t$, we assume that the addition in the indices, such as $v_{i+t}$, is taken modulo $k$.

\section{Tightness examples}\label{sec:example}

\begin{exa}\label{tight_gen}
For all $k\geq 2$ and all $n$ divisible by $k$, there exists a family $\mathcal{F}$ of balanced $k$-partite graphs on $n$ vertices such that for all $F\in \mathcal{F}$, 
$$\delta(F)\geq \dk-1,$$ but $F$ does not have a Hamiltonian cycle.  
\end{exa}

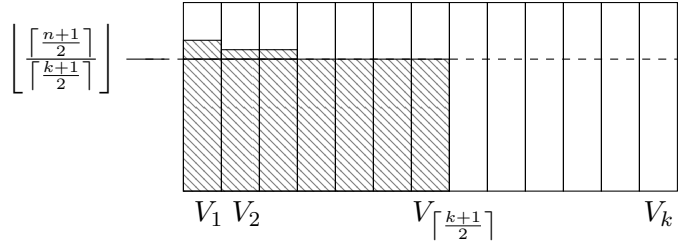
\begin{figure}[ht]
\centering
\begin{tikzpicture}[scale = .25]
\draw[pattern=north west lines, pattern color=gray] (-18, 1) rectangle (-16, 2);
\draw[pattern=north west lines, pattern color=gray] (-16, 1) rectangle (-12, 1.5);
\draw[pattern=north west lines, pattern color=gray] (-18, -6) rectangle (-4, 1);
\draw (-21, 1) node[left]{$\floor{\frac{\ceiling{\frac{n+1}{2}}}{\ceiling{\frac{k+1}{2}}}}$} -- (-19, 1);
\draw[dashed] (-20, 1) -- (8, 1);
\draw (-18, -6) node[below right]{$V_1$} rectangle (-16, 4);
\draw (-16, -6) node[below right]{$V_2$} rectangle (-14, 4);
\draw (-14, -6) rectangle (-12, 4);
\draw (-12, -6) rectangle (-10, 4);
\draw (-10, -6) rectangle (-8, 4);
\draw (-8, -6) rectangle (-6, 4);
\draw (-6, -6) node[below]{$~~~~~~~~~V_{\ceiling{\frac{k+1}{2}}}$} rectangle (-4, 4);
\draw (-4, -6) rectangle (-2, 4);
\draw (-2, -6) rectangle (0, 4);
\draw (0, -6) rectangle (2, 4);
\draw (2, -6) rectangle (4, 4);
\draw (4, -6) rectangle (6, 4);
\draw (6, -6) node[below]{$~~~~V_k$} rectangle (8, 4);
\end{tikzpicture}
\caption{The family of graphs $\mathcal{F}$.  The shaded sets represent $X_1, \dots, X_{\ceiling{\frac{k+1}{2}}}$.}
\end{figure}

\begin{proof}
First note that
\begin{equation}\label{equal}
\floor{\frac{n+2}{2\ceiling{\frac{k+1}{2}}}}=\floor{\frac{\ceiling{\frac{n+1}{2}}}{\ceiling{\frac{k+1}{2}}}}.
\end{equation}
Since if $k$ is even, then both sides of the equation equal $\floor{\frac{n+2}{k+2}}$; if $k$ is odd and $n$ is even, then both sides of the equation equal $\floor{\frac{n+2}{k+1}}$; and if $k$ is odd and $n$ is odd then we get that $\floor{\frac{n+1}{k+1}} = \floor{\frac{n+2}{k+1}}$, which is true since $\frac{n+2}{k+1}$ is not an integer. 

Let $\mathcal{F}$ be the family of graphs which can be obtained from a complete $k$-partite graph with parts $V_1, \dots, V_k$ such that $|V_i|=\frac{n}{k}$ for all $i\in [k]$, by selecting some $X_i\subseteq V_i$ for all $i\in [\ceiling{\frac{k+1}{2}}]$ such that $|X_1|\geq \dots\geq |X_{\ceiling{\frac{k+1}{2}}}|=\floor{\frac{\ceiling{\frac{n+1}{2}}}{\ceiling{\frac{k+1}{2}}}}$ and $|X_1\cup\dots\cup X_{\ceiling{\frac{k+1}{2}}}|=\ceiling{\frac{n+1}{2}}$.  Add all edges between parts except for those between a vertex in $X_i$ and $X_j$ for all $i, j \in [\ceiling{\frac{k+1}{2}}]$. Note that every $F\in \mathcal{F}$ has an independent set of size $\ceiling{\frac{n+1}{2}}$ and thus does not contain a Hamiltonian cycle.

Finally to see that the degree condition is satisfied, let $i\in [\ceiling{\frac{k+1}{2}}]$ and let $v\in X_{i}$.  We have by \eqref{equal} 
\begin{align*}
d(v)&=(1-\frac{1}{k})n-|X_1\cup \dots X_{i-1}\cup X_{i+1}\cup  \dots\cup X_{\ceiling{\frac{k+1}{2}}}|\\
&\geq (1-\frac{1}{k})n-\left(\ceiling{\frac{n+1}{2}}-\floor{\frac{\ceiling{\frac{n+1}{2}}}{\ceiling{\frac{k+1}{2}}}}\right)
=\dk-1. \qedhere
\end{align*}
\end{proof}

\begin{exa}\label{n=2k} Let $n\geq 8$ be divisible by $4$ and let $k=\frac{n}{2}$.
\begin{enumerate}[label=\emph{(\roman*)}]
\item\label{e1}  There exists a family $\mathcal{F}_1$ of balanced $k$-partite graphs on $n$ vertices such that for all $F_1\in \mathcal{F}_1$, $\delta(F_1)\geq \frac{n}{2}+\floor{\frac{n+2}{k+2}}-\frac{n}{k}=\frac{n}{2}-1$, but $\kappa(F_1)\leq 1$ and thus $F_1$ does not have a Hamiltonian cycle.

\item\label{e2} There exists a 2-connected balanced 4-partite graph $F_2$ on 8 vertices with $\alpha(F_2)=3$ such that $F_2$ does not have a Hamiltonian cycle.

\item\label{e3} There exists a family $\mathcal{F}_3$ of balanced $k$-partite graphs on $n$ vertices such that for all $F_3\in \mathcal{F}_3$, $\delta(F_3)\geq \frac{n}{2}+\floor{\frac{n+2}{k+2}}-\frac{n}{k}=\frac{n}{2}-1$, $\kappa(F_2)\geq 2$, and $\alpha(F_3)=\frac{n}{2}$, but $F_3$ does not have a Hamiltonian cycle.
\end{enumerate}
\end{exa}

\begin{figure}[ht]
\centering
\includegraphics[scale=1]{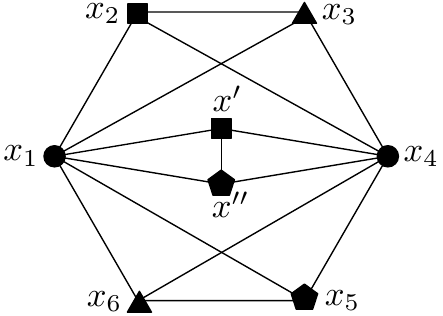}
\caption{The graph $F_2$.}\label{fig_F2}
\end{figure}

\begin{proof}
\begin{enumerate}
\item Let $V_i=\{x_i, y_i\}$ for all $i\in [k]$.  Add all edges inside $\{x_1, \dots, x_k\}$, add all edges from $y_k$ to $\{y_1, \dots, y_{k-1}\}$, and for all $i\in [k-1]$ add at least $k-1$ edges from $y_i$ to $\{y_1, \dots, y_{i-1}, y_{i+1}, \dots, y_{k-1}\}\cup \{x_{k}\}$.  Let $\mathcal{F}_1$ be the family of graphs thus obtained.  Note that every graph $F_1\in \mathcal{F}_1$ has $\delta(F_1)=k-1=\frac{n}{2}-1$ and $\kappa(F_1)\leq 1$ and thus $F_1$ does not have a Hamiltonian cycle.  

\item We let $F_2$ be the graph in Figure \ref{fig_F2} which can be seen to be a balanced 4-partite graph (with vertices of the same shape being in the same part of the partition) which is 2-connected and has $\alpha(F_2)=3$.  Note that $F_2$ has no Hamiltonian cycle since $G-x_1-x_4$ has three components.

\item Let the parts be labeled $X_1, \dots, X_{k/2}$, $Y_1, \dots, Y_{k/2}$ and let $X=\cup_{i=1}^{k/2} X_i$ and $Y=\cup_{i=1}^{k/2}Y_i$.  Let $y'\in Y_{k/2}$, let $y''\in Y\setminus \{y'\}$, and let $x'\in X$. Add all edges between $X$ and $Y\setminus \{y', y''\}$, all edges from $y''$ to $X\setminus \{x'\}$, and all edges from $y'$ to $\{x'\}\cup (Y\setminus Y_{k/2})$.  Furthermore, we may add any number of other edges between the parts $Y_1, \dots, Y_{k/2}$ and we may add the edge $x'y''$.  Let $\mathcal{F}_3$ be the family of graphs thus obtained.  Let $F_3\in \mathcal{F}_3$ and let $H$ be the bipartite graph induced by $[X,Y]$.  It is easily seen that $\delta(F_3)\geq \frac{n}{2}-1$, $\kappa(F_2)\geq 2$, and $\alpha(F_3)=\frac{n}{2}$.  Since $X$ is an independent set, if $F_3$ has a Hamiltonian cycle, it must be in $H$; however, since $y'$ has degree 1 in $H$, there is no Hamiltonian cycle in $H$.  
\end{enumerate}
\end{proof}

\section{General lemmas}\label{sec:lemmas}

In this section we state three general results which are useful for finding Hamiltonian cycles, beginning with two classics.

\begin{thm}[Dirac \cite{D}]\label{dir}
Let $n\geq d\geq 3$.  If $G$ is 2-connected and $\delta(G)\geq d/2$, then $G$ has a cycle of length at least $d$.  
\end{thm}

\begin{thm}[Chv\'atal \cite{C}]\label{chv}
Let $G=(U,V,E)$ be a bipartite graph on $n\geq 4$ vertices with vertex sets $U=\{u_1, \dots, u_{n/2}\}$ and $V=\{v_1, \dots, v_{n/2}\}$.  If for all $1\leq k<n/2$, $$d(v_k)\leq k \Rightarrow d(u_{n/2-k})\geq \frac{n}{2}-k+1,$$ 
then $G$ has a Hamiltonian cycle.
\end{thm}

The main lemma which we use to begin the proof of Theorem \ref{main} and Proposition \ref{prop:n=2k} is the following combination of the well known result of Nash-Williams \cite{NW} and a (slight weakening of a) result of Bauer, Veldman, Morgana, Schmeichel \cite{BVMS}.  We provide a proof for completeness.

We say that a cycle $C$ in a graph $G$ is \emph{strongly dominating} if $V(G)\setminus V(C)$ is an independent set and no two vertices of $\bigcup_{u\in V(G)\setminus V(C)}N(u)$ appear consecutively on $C$.

\begin{lem}[see {\cite[Lemmas 1,2,3,4]{NW}} and {\cite[Lemma 8]{BVMS}}]\label{domcycle}
Let $G$ be a graph on $n$ vertices.  If $G$ is 2-connected and $\delta(G)\geq \frac{n+2}{3}$, then every longest cycle of $G$ is strongly dominating.
\end{lem}

\begin{proof}
Let $C=v_1v_2\dots v_kv_1$ be a longest cycle in $G$ and let $P=u_1u_2\dots u_r$ be a longest path in $G-C$.  If $r\leq 1$, then we are done; so suppose $r\geq 2$.  We have $k+r\leq n$ and note that by Theorem \ref{dir}, we have $k\geq 2\delta(G)\geq \frac{2n+4}{3}$ and thus 
\begin{equation}\label{nrlower}
3r+4\leq n.  
\end{equation}
Let $X=N(u_1)\cap V(C)$ and $Y=N(u_r)\cap V(C)$ and note that 
\begin{equation}\label{XYlower}
|X|, |Y|\geq \delta(G)-(r-1).
\end{equation}

The key observation is that by the maximality of $C$, no two vertices in $X\cup Y$ are consecutive along $C$, and furthermore if $v_i\in X\cap Y$, then none of $v_{i-r}, \dots, v_{i-1}, v_{i+1}, \dots, v_{i+r}$ are in $X\cup Y$.

First suppose that $X\subseteq Y$ or $Y\subseteq X$; without loss of generality $X\subseteq Y$.  In this case we have by \eqref{XYlower}, 
\begin{equation}\label{Xineq}
n-r\geq k\geq (r+1)|X|\geq (r+1)(\delta(G)-(r-1)).
\end{equation}

First suppose $r=2$, in which case \eqref{Xineq} becomes $n-2\geq 3\left(\frac{n+2}{3}-1\right)=n-1$, a contradiction.  Now suppose $r\geq 3$ in which case \eqref{Xineq} becomes $$n\leq \frac{3r^2-5r-5}{r-2}=3r+1-\frac{3}{r-2},$$ contradicting \eqref{nrlower}.  

Now suppose that $X\setminus Y\neq \emptyset$ and $Y\setminus X\neq \emptyset$.  There are vertices $v_i, v_j\in V(C)$ with the following properties: $v_i\in X\setminus Y$ and the next vertex $v_{i'}\in X\cup Y$ which appears after $v_i$ satisfies $v_{i'}\in Y$ (meaning that $i'\geq i+r+1$), and $v_j\in Y\setminus X$ and the next vertex $v_{j'}\in X\cup Y$ from $X\cup Y$ which appears after $v_j$ satisfies $v_{j'}\in X$ (meaning that $j'\geq j+r+1$).  Each vertex of $((X\setminus Y)\cup (Y\setminus X))\setminus \{v_i, v_j\}$ is followed by at least one vertex from $V(C)\setminus (X\cup Y)$ and each vertex of $(X\cap Y)\cup \{v_i, v_j\}$ is followed by at least $r$ vertices from $V(C)\setminus (X\cup Y)$.  So we have 
\begin{align*}
n-r\geq k&\geq 2(|X\setminus Y|+|Y\setminus X|-2)+(r+1)(|X\cap Y|+2)\\
&=|X|+|Y|+|X\cup Y|+(r-2)|X\cap Y|+2(r-1)\\
&\geq \min\{4\delta(G)-2(r-1), 3\delta(G)-1\}\\
&\geq \min\{4(n+2)/3-2(r-1), n+1\},
\end{align*}
where the second to last inequality is seen by using \eqref{XYlower} and splitting into cases whether $|X\cap Y|=0$ or not.  However, $4(n+2)/3-2(r-1)\leq n-r$ implies $n\leq 3r-14$, contradicting \eqref{nrlower}.

To see that the second part of the definition of strongly dominating is satisfied, suppose that $C=v_1\dots v_kv_1$ is a longest cycle and suppose $V(G)\setminus V(C)=\{u_1, \dots, u_r\}$ is an independent set.  If $|V(G)\setminus V(C)|\leq 1$, we are done, so suppose $r\geq 2$.  Let $X=N(u_1)\cap V(C)$ and suppose (without loss of generality) for contradiction that $v_1\in X$ and $v_2\in N(u_2)$.  By the maximality of $C$, this implies that $v_3\not \in N(u_2)$ and for all $i\geq 3$, if $v_i\in X$, then $v_{i+1}, v_{i+2}\not\in N(u_2)$.  Since $k\leq n-2$, this implies that 
$$\frac{n+2}{3}\leq |N(u_2)|\leq k-(2|N(u_1)|-1)\leq n-1-2\left(\frac{n+2}{3}\right)=\frac{n-7}{3},$$ a contradiction.
\end{proof}

\section{Filtering out $\mathcal{F}_1$ and $F_2$}\label{sec:filter}

We want to say that every balanced $k$-partite graph satisfying \eqref{cf} is both 2-connected and every longest cycle in strongly dominating; however, there are two exceptions and we deal with those exceptions before beginning the main proof in next Section.  

First, we show that every balanced $k$-partite graph satisfying \eqref{cf} is either $2$-connected or belongs to the family $\mathcal{F}_1$ in Example \ref{n=2k}.\ref{e1}.

\begin{lem}\label{F1}
Let $k\geq 3$, let $n$ be an integer such that $n\geq 2k$, and let $G$ be a balanced $k$-partite graph on $n$ vertices.  If $$\delta(G) \geq \dk,$$ then $G$ is 2-connected unless $n=2k$ and $G\in \mathcal{F}_1$ (see Example \ref{n=2k}.\ref{e1}).
\end{lem}

\begin{proof} Let $V_1, \dots, V_k$ be the parts of $G$.  Suppose for contradiction that $\kappa(G)\leq 1$.  Let $A, B, C$ be a partition of $V(G)$ such that $|C|\leq 1$ and $G-C$ is not connected.

First suppose that there exists $i\in [k]$ such that $V_i\subseteq A\cup C$ or $V_i\subseteq B\cup C$.  Without loss of generality suppose $V_i\subseteq B\cup C$ and let $u\in A$ and $v\in B\cap V_i$.  We have 
$$2\delta(G)\leq d(u)+d(v)\leq |A|+|C|-1+|B|+|C|-\frac{n}{k}=(1-\frac{1}{k})n+|C|-1\leq (1-\frac{1}{k})n,$$
contradicting Fact \ref{ineq}.\ref{f1} when $k$ is even and $n\geq 3k$, and contradicting Fact \ref{ineq}.\ref{f2} when $k$ is odd and $n\geq 2k$.  

So unless $k$ is even and $n=2k$, we must have that for all $i\in [k]$, $V_i\cap A\neq \emptyset$ and $V_i\cap B\neq \emptyset$.  Either $C=\emptyset$ and we let $u\in A$ and $v\in B$, or $C\neq \emptyset$ and suppose without loss of generality that $V_1\cap C\neq \emptyset$ in which case we let $u\in V_1\cap A$ and $v\in V_1\cap B$.  Either way we have 
$$2\delta(G)\leq d(u)+d(v)\leq n-\frac{n}{k},$$
contradicting Fact \ref{ineq}.\ref{f1} when $k$ is even and $n\geq 3k$, and contradicting Fact \ref{ineq}.\ref{f2} when $k$ is odd and $n\geq 2k$.

Finally, suppose $k$ is even and $n=2k$ which implies $\delta(G)\geq \frac{n}{2}-1$.  For all $u\in A$ we have $\frac{n}{2}-1\leq d(u)\leq |A|+|C|-1$ which implies $|A|\geq \frac{n}{2}-|C|$ and for all $v\in B$ we have $\frac{n}{2}-1\leq d(v)\leq |B|+|C|-1$ which implies $|B|\geq \frac{n}{2}-|C|$.  If $C=\emptyset$, this implies $|A|=\frac{n}{2}$ and $|B|=\frac{n}{2}$.  If $C\neq \emptyset$, we have $|A|\geq \frac{n}{2}-1$ and $|B|\geq \frac{n}{2}-1$, so without loss of generality suppose $|A|+|C|=\frac{n}{2}$ and $|B|=\frac{n}{2}$.

If there exists $i\in [k]$ such that $V_i\subseteq A\cup C$ or $V_i\subseteq B$; say $V_i\subseteq A\cup C$, then for $u\in V_i\cap A$, we have 
\[\frac{n}{2}-1\leq d(u)\leq |A|+|C|-|V_i|=\frac{n}{2}-2,\] 
a contradiction.  Thus for all $i\in [k]$, we have $|V_i\cap (A\cup C)|=1$ and $|V_i\cap B|=1$.  Let $V_i=\{x_i,y_i\}$ for all $i\in [k]$ and let $X=\{x_1, \dots, x_k\}$ and $Y=\{y_1, \dots, y_k\}$ and suppose $X=A\cup C$ and $Y=B$.  So it must be the case that every vertex $u\in A$ is adjacent to precisely the vertices in $X\setminus \{u\}$ which means $G[X]$ is a clique.  Also every vertex $v\in Y$ is adjacent to at least $\frac{n}{2}-1$ of the $\frac{n}{2}$ vertices in $(Y\cup C)\setminus \{v\}$.  Thus $G\in \mathcal{F}_1$.  
\end{proof}

We now prove a lemma which shows that when $G$ is a 2-connected balanced $k$-partite graph satisfying \eqref{cf}, we either have that every longest cycle in $G$ is strongly dominating or $G$ is isomorphic to the graph $F_2$ in Example \ref{n=2k}.\ref{e2}.

\begin{lem}\label{F2}
Let $k\geq 3$ and let $n\geq 2k$ and let $G$ be a balanced $k$-partite graph on $n$ vertices.  If $G$ is 2-connected and $\delta(G)\geq \dk$, then every longest cycle of $G$ is strongly dominating unless $n=8$ and $G\cong F_2$ (see Example \ref{n=2k}.\ref{e2}).
\end{lem}

\begin{proof}
We will show that, unless $n=8$ and $k=4$, we have $\delta(G)\geq \dk\geq \frac{n+2}{3}$ and thus we are done by Lemma \ref{domcycle}.

First suppose $k$ is odd, in which case $\dk=\ceiling{\frac{n}{2}}+\floor{\frac{n+2}{k+1}}-\frac{n}{k}$.  First note that when $k=3$ and $n=6$ or $n=9$ we have by direct inspection that $\ceiling{\frac{n}{2}}+\floor{\frac{n+2}{k+1}}-\frac{n}{k}\geq \frac{n+2}{3}$.  So in the remaining cases we have 
\begin{equation}\label{n10}
n\geq 10\geq 10-\frac{24k-60}{k^2+k-6}=\frac{2k(5k-7)}{k^2+k-6}.
\end{equation}
Thus, using Fact \ref{kodd}, we have
\begin{align*}
\ceiling{\frac{n}{2}}+\floor{\frac{n+2}{k+1}}-\frac{n}{k}-\frac{n+2}{3}&\geq \frac{n}{2}+\frac{n+2-(k-1)}{k+1}-\frac{n}{k}-\frac{n+2}{3}\\
&=\left(\frac{1}{6}-\frac{1}{k(k+1)}\right)n-\frac{2}{3}-\frac{k-3}{k+1}\\
&\stackrel{\eqref{n10}}{\geq} \left(\frac{1}{6}-\frac{1}{k(k+1)}\right)\frac{2k(5k-7)}{k^2+k-6}-\frac{2}{3}-\frac{k-3}{k+1}=0,
\end{align*}
as desired.

Now suppose $k$ is even, in which case $\dk=\frac{n}{2}+\floor{\frac{n+2}{k+2}}-\frac{n}{k}$.  Note that aside from the case $n=8$ and $k=4$ we have
\begin{equation}\label{n12}
n\geq 12\geq 12-\frac{2(k+18)(k-4)}{k^2+2k-12}=\frac{2k(5k-2)}{k^2+2k-12}.
\end{equation}
Thus
\begin{align*}
\frac{n}{2}+\floor{\frac{n+2}{k+2}}-\frac{n}{k}-\frac{n+2}{3}&\geq \frac{n}{2}+\frac{n+2-k}{k+2}-\frac{n}{k}-\frac{n+2}{3}\\
&=\left(\frac{1}{6}-\frac{2}{k(k+2)}\right)n-\frac{2}{3}-\frac{k-2}{k+2}\\
&\stackrel{\eqref{n12}}{\geq} \left(\frac{1}{6}-\frac{2}{k(k+2)}\right)\frac{2k(5k-2)}{k^2+2k-12}-\frac{2}{3}-\frac{k-2}{k+2}=0,
\end{align*}
as desired.

Finally suppose $n=8$ and $k=4$ and let $C$ be a longest cycle of $G$.  Since $G$ is 2-connected and $\delta(G)\geq \frac{8}{2}+\floor{\frac{10}{6}}-\frac{8}{4}=3$, Theorem \ref{dir} implies that $C$ has length at least $6$.  If $C$ had length at least 7, it would be a strongly dominating cycle, so suppose $C$ has length 6.  Let $C=x_1x_2x_3x_4x_5x_6$ and let $x'$ and $x''$ be the two vertices in $V(G)\setminus V(C)$.  If $x'x''\not\in E(G)$, then by the maximality of $C$ it is easily seen that, without loss of generality, $N(x')=N(x'')=\{x_1,x_3,x_5\}$ and thus $C$ is strongly dominating; so suppose that $x'x'' \in E(G)$.  

Without loss of generality suppose $x'x_1 \in E(G)$. If either $x''x_2\in E(G)$ or $x''x_6\in E(G)$, then $G$ has a Hamiltonian cycle; and if either $x''x_3\in E(G)$ or $x''x_5\in E(G)$, then $G$ has a cycle longer than $C$, a contradiction.  Since $\delta(G)\geq 3$, this forces $x''x_1,x''x_4 \in E(G)$.  By the same argument we get $x'x_4 \in E(G)$.  If $x_6x_3 \in E(G)$, then $x_6x_3x_2x_1x'x''x_4x_5x_6$ is a Hamiltonian cycle, so $x_6x_3 \not \in E(G)$, and by symmetry $x_2x_5 \not \in E(G)$. If $x_6x_2 \in E(G)$, then $x_6x_2x_3x_4x'x''x_1x_6$ is a cycle longer than $C$, a contradiction. So $x_6x_2 \not \in E(G)$ and by symmetry $x_3x_5 \not \in E(G)$. Since $\delta(G)\geq 3$, this forces $x_6x_4,x_5x_1,x_2x_4,x_3x_1 \in E(G)$. Therefore $G \cong \mathcal{F}_2$. 
\end{proof}

\section{Proof of Theorem \ref{main} and Proposition \ref{prop:n=2k}}\label{sec:main}

Let $k\geq 3$ and let $G$ be a balanced $k$-partite graph on $n$ vertices.  Let $V_1,V_2, \dots, V_k$ denote the parts and note that $|V_i|=\frac{n}{k}=:m $ for all $i\in [k]$.  Since the case $k=n$ is Dirac's theorem, we suppose $k\leq \frac{n}{2}$ and since the case $k=2$ is handled in \cite{MM} and \cite{FJP}, we suppose $k\geq 3$.  Furthermore, if $k=\frac{n}{2}$, we suppose that $G\not\in \mathcal{F}_1$ and $G\not\cong F_2$ (see Example \ref{n=2k}).  Now let $C$ be a maximum length cycle and suppose for contradiction that $C$ is not Hamiltonian.  By Lemma \ref{F1} and Lemma \ref{F2} we may assume that $C$ is strongly dominating.  

Without loss of generality, let 
$z\in V_1\setminus V(C)$.

Let $S=(V(G)\setminus V(C))\cup \{v_{i+1}: v_i\in N(z)\}$ and $R=(V(G)\setminus V(C))\cup \{v_{i-1}: v_i\in N(z)\}$ and note that 
\begin{equation} \label{SR}
|S|, |R|\geq \delta(G)+1.
\end{equation}  
Since $C$ is strongly dominating, both $S$ and $R$ are independent sets. 
For each $i \in [k]$, set $$S_i = S
\cap V_i \text{ and } R_i = R \cap V_i.$$ 
Define $\ell=|\{i\in [k]: S_i\neq \emptyset\}|$ and $\ell'=|\{i\in [k]: R_i\neq \emptyset\}|$ and without loss of generality suppose $$\ell\leq \ell'.$$  Furthermore, without loss of generality, we may suppose that 
\begin{align*}
S_i \neq \emptyset \text{ for all } i\in [\ell] \text{ and } S_j=\emptyset \text{ for all } j\in [k]\setminus [\ell].
\end{align*}

\begin{cla}\label{elllower}
$\ell, \ell'\geq \ceiling{\frac{k}{2}}$
\end{cla}

\begin{proof}
We claim that $|R|,|S|\geq \delta(G)+1>(\ceiling{\frac{k}{2}}-1)\frac{n}{k}$, which implies the result.  Indeed, we have 
\begin{equation}\label{k/2}
\delta(G) + 1 - (\ceiling{\frac{k}{2}} - 1)\frac{n}{k} \geq \ceiling{\frac{n}{2}} + \floor{\frac{n+2}{2\ceiling{\frac{k+1}{2}}}} + 1 - \frac{n}{k}\ceiling{\frac{k}{2}}.
\end{equation}
When $k$ is even, \eqref{k/2} reduces to $\floor{\frac{n+2}{k+2}} + 1 > 0$, and when $k$ is odd, by Fact \ref{kodd}, \eqref{k/2} reduces to  $\frac{n}{2}+\frac{n}{k+1}-\frac{k-3}{k+1}-\frac{n}{k}\left(\frac{k+1}{2}\right)+1=\frac{n}{k+1}-\frac{n}{2k}+1-\frac{k-3}{k+1}=\frac{(k-1)n}{2k(k+1)}+\frac{4}{k+1}>0$.
\end{proof}

\begin{cla}\label{Si}~
\begin{enumerate}[label=\emph{(\roman*)}]
\item \label{c1} For all $y \in S$, 
$|\overline {N(y)} \setminus  S| \leq n-2\delta(G)-1.$  For all $y \in R$, 
$|\overline {N(y)}  \setminus R| \leq n-2\delta(G)-1.$

\item \label{c2} For all $i\in [k]$, if $S_i\neq \emptyset$, then $|S_i| \geq  2\delta(G)+1-(1-\frac{1}{k})n\geq \frac{1}{2}(\frac{n}{k}-(\floor{\frac{n-1}{2}}-\delta(G))).$  For all $i\in [k]$, if $R_i\neq \emptyset$, then $|R_i| \geq  2\delta(G)+1-(1-\frac{1}{k})n\geq \frac{1}{2}(\frac{n}{k}-(\floor{\frac{n-1}{2}}-\delta(G))).$
\end{enumerate}
\end{cla}

\begin{proof}
\begin{enumerate}
\item Since $C$ is a longest cycle of $G$, the vertex subsets
$N(y)$, $S$, and $\overline{N(y)}  \setminus S$ are pairwise disjoint for all $y \in S$.
Thus
$ n= |N(y)| + |S| + |\overline{N(y)}  \setminus S|\geq 2\delta(G)+1+|\overline{N(y)}  \setminus S|$, where the inequality holds by \eqref{SR}.  Thus $|\overline {N(y)}  \setminus S| \leq n-2\delta(G)-1.$  Similarly, $N(y)$, $R$, and $\overline{N(y)}  \setminus R$ are pairwise disjoint for all $y \in R$, so $|\overline{N(y)}  \setminus R| \leq n - 2\delta(G) - 1$

\item Let $y\in S_i$.  We have that $V_i\setminus S_i\subseteq \overline {N(y)}  \setminus S$ so by (i) we have that $|S_i|\geq \frac{n}{k}-|\overline {N(y)}  \setminus S|\geq 2\delta(G)+1-(1-\frac{1}{k})n$ as desired.   Similarly, if $y \in R_i$, then by (i) we have that $|R_i|\geq \frac{n}{k} - |\overline{N(y)}  \setminus R| \geq 2\delta(G) + 1 - (1 - \frac{1}{k})n$.

Finally, we have $2\delta(G)+1-(1-\frac{1}{k})n\geq \frac{1}{2}(\frac{n}{k}-(\floor{\frac{n-1}{2}}-\delta(G)))$ by Fact \ref{ineq}.\ref{f3}.\qedhere
\end{enumerate}
\end{proof}

\begin{cla}\label{zVi}~
For all $i\in [k]$, if $S_i\cap R_i\neq \emptyset$, then $|N(z)\cap V_i| \leq \floor{\frac{n-1}{2}}-\delta(G).$  
\end{cla}

\begin{proof}

Let $2\leq i\leq k$ such that $S_i\cap R_i\neq \emptyset$ and let $y \in S_i\cap R_i$. So $y$ is a successor along $C$ of some vertex 
in $N(z)$, and a predecessor along $C$ of some vertex in $N(z)$ as well. 
Since $C$ is a longest cycle of $G$, neither $N(z)$ nor $N(y)$ contains
two consecutive vertices of $C$, so $N(y)\cap (S\cup R)
= \emptyset$. Thus, 
$$n-1\geq  |V(C)| \geq 2 |N(z)\cup N(y)| = 2(d(y) + |N(z)  \setminus N(y)|)
\geq 2(d(y) +|N(z)\cap V_i|).$$
Rearranging gives the result.
\end{proof}

\begin{cla}\label{ellupper}
$\frac{\ell+\ell'}{2}< \ceiling{\frac{k+1}{2}}$
\end{cla}

\begin{proof}
Let $i_1\leq i_2\leq \dots\leq i_{\ell'}$ be the indices such that $R_{i_j}\neq \emptyset$ for all $j\in [\ell']$.  By Claim \ref{Si}.\ref{c2} and  Claim \ref{zVi} and the fact that $z\in S_1\cap R_1$ we see that each of the sets $S_2, \dots, S_\ell$, $R_{i_2}, \dots, R_{i_{\ell'}}$ contributes at least $\frac{1}{2}\left(\frac{n}{k}-\left(\floor{\frac{n-1}{2}}-\delta(G)\right)\right)$ to $|\overline{N(z)} \setminus V_1|$.  So we have
$$
\delta(G)\leq d(z) \leq (1-\frac{1}{k})n-\frac{1}{2}\left(\frac{n}{k}-\left(\floor{\frac{n-1}{2}}-\delta(G)\right)\right)(\ell+\ell'-2).
$$

Solving the above inequality for $\frac{\ell+\ell'}{2}$, we have 
$$
\frac{\ell+\ell'}{2}\leq  \frac{\ceiling{\frac{n+1}{2}}}{\floor{\frac{n+2}{2\ceiling{\frac{k+1}{2}}}}+1}\leq \frac{\frac{n+2}{2}}{\frac{n+2}{\tck}-\frac{\tck-1}{\tck}+1}=\frac{n+2}{n+3}\ck<\ck,$$
as desired.  
\end{proof}

Since we are supposing without loss of generality that $\ell\leq \ell'$, we have by Claim \ref{elllower} and Claim \ref{ellupper}  that $$\ceiling{\frac{k}{2}}\leq \ell<\ceiling{\frac{k+1}{2}}.$$  Thus if $k$ is odd, we have a contradiction.  So for the rest of the proof we will suppose that $k$ is even and consequently by Claim \ref{elllower} and \ref{ellupper}, we have $\ell=\frac{k}{2}$.  

\subsection{$k$ is even and $\ell=\frac{k}{2}$}

Let $$A = \bigcup_{i=1}^{\ell}V_i ~\text{ and }~ B = \bigcup_{i= \ell+1}^kV_i,$$ and let $H$ be the bipartite graph induced by $[A,B]$. Label the vertices of $A$ as $u_1, \dots, u_{n/2}$ such that $d_H(u_1)\leq \dots\leq d_H(u_{n/2})$ and label the vertices of $B$ as $v_1, \dots, v_{n/2}$ such that $d_H(v_1)\leq \dots\leq d_H(v_{n/2})$.  Recall that $S\subseteq A$.

Since we are in the case where $k$ is even, \eqref{cf} reduces to $$\delta(G)\geq \frac{n}{2}+\floor{\frac{n+2}{k+2}}-\frac{n}{k}.$$ 


\begin{cla}\label{AB}~
\begin{enumerate}[label=\emph{(\roman*)}]
\item\label{SB} $\delta(S, B)\geq \frac{n}{2}+2\floor{\frac{n+2}{k+2}}-\frac{2n}{k}+1$, with equality only if $S_i=V_i$ for some $i\in [\ell]$ 
\item\label{ASB} $\delta(A\setminus S, B)\geq \floor{\frac{n+2}{k+2}}\geq |A\setminus S|+1$, 
\item\label{BA} $\delta(B,A)\geq \floor{\frac{n+2}{k+2}}\geq |A\setminus S|+1$.
\end{enumerate}
\end{cla}

\begin{proof}
\begin{enumerate}
\item This follows from  Claim \ref{Si}.\ref{c1} since for all $y\in S$,  $$d(y, B)\geq |B|-|\overline{N(y)} \setminus S|\geq |B|-(n-2\delta(G)-1)\geq \frac{n}{2}+2\floor{\frac{n+2}{k+2}}-\frac{2n}{k}+1.$$
If we have equality above, this implies that $\overline{N(y)} \setminus S\subseteq B$, which in particular implies that if $y\in V_i$, then $V_i\setminus S_i=\emptyset$.

\item We have 
\begin{align*}
\delta(A\setminus S, B)&\geq \delta(G)-(\frac{n}{2}-\frac{n}{k})\\
&\geq \floor{\frac{n+2}{k+2}}
\geq \frac{n}{k}-\floor{\frac{n+2}{k+2}}
\geq \frac{n}{2}-\delta(G)\stackrel{\eqref{SR}}{\geq}  \frac{n}{2}-|S|+1=|A\setminus S|+1,
\end{align*}
where the third inequality holds by Fact \ref{ff}.\ref{ff1}. \qedhere

\item Since $\delta(B,A)\geq \delta(G)-(\frac{n}{2}-\frac{n}{k})$, the rest of the calculation is the same as in  \ref{ASB}.
\end{enumerate}
\end{proof}

Before proceeding with the rest of the proof, we finally filter out $\mathcal{F}_3$.

\begin{cla}\label{n/2-1}
If $\delta(G)\geq \frac{n}{2}-1$, then either $G$ has a Hamiltonian cycle or $k=\frac{n}{2}$ and $G\in \mathcal{F}_3$.
\end{cla}

\begin{proof}
By \eqref{SR} and the fact that $\ell=\frac{k}{2}$, we have $|S|=\frac{n}{2}$ and thus $A=S$ which means $A$ is an independent set.  So by Claim \ref{AB}.\ref{SB}, we have $\delta(A,B)\geq \frac{n}{2}-1$.  Furthermore we have by Claim \ref{AB}.\ref{BA} that $\delta(B,A)\geq \floor{\frac{n+2}{k+2}}$.  If $\delta(B,A)\geq 2$, then by Theorem \ref{chv}, $G$ has a Hamiltonian cycle.  

So suppose $\delta(B,A)=1$ which implies $n=2k$.  In this case there is a vertex $y' \in B$ such that $y'$ only has one neighbor in $A$, say $x'$.  We have $\delta(A,B)\geq \frac{n}{2}-1=|B|-1$ so every vertex in $A\setminus \{x'\}$ is adjacent to every vertex in $B\setminus \{y'\}$.  Since $d(y', A)=1$ and $d(y')=\frac{n}{2}-1$, it must be the case that $y'$ is adjacent to everything in $B$ except the other vertex in its own part.  Now we have all the edges between $A \setminus \{x'\}$ and $B \setminus \{y'\}$, the edge $x'y'$, all the edges from $y'$ to $B$ excluding the vertex in its own part, and we have all but possibly one edge from $x'$ to $B \setminus \{y'\}$, so $G \in \mathcal{F}_3$.
\end{proof}

Now for the rest of the proof we may suppose that $n\geq 3k$ (i.e. $m\geq 3$).  We now use Theorem \ref{chv} to show that $H$, and therefore $G$, has a Hamiltonian cycle.   

Suppose there exists $i\in [\frac{n}{2}]$ such that $d_H(v_i)\leq i$.  By Claim \ref{AB}.\ref{BA}, we must have 
\begin{equation}\label{i}
i\geq \delta(B, A)\geq \floor{\frac{n+2}{k+2}}.
\end{equation}

\noindent
\tbf{Case 1} ($\frac{n}{2}-i\leq |A\setminus S|$)  By Claim \ref{AB}.\ref{ASB} we have $d_H(u_{\frac{n}{2}-i})\geq \delta(A\setminus S, B) \geq |A\setminus S|+1\geq \frac{n}{2}-i+1$.

\noindent
\tbf{Case 2} ($\frac{n}{2}-i\geq |A\setminus S|+1$)

\tbf{Case 2.1} ($k\geq 6$) Note that since $k\geq 6$, when $3k\leq n\leq 5k$, we have $\delta(G)\geq \frac{n}{2}-1$ and thus we are done by Claim \ref{n/2-1}.  So for remainder of this case, suppose $n\geq 6k$ (i.e. $m\geq 6$).

By Claim \ref{AB}.\ref{SB} $$d_H(u_{\frac{n}{2}-i})\geq \delta(S, B)\geq \frac{n}{2}+2\floor{\frac{n+2}{k+2}}-\frac{2n}{k}+1\geq \frac{n}{2}-\floor{\frac{n+2}{k+2}}+1\stackrel{\eqref{i}}{\geq} \frac{n}{2}-i+1,$$
where the third inequality holds by Fact \ref{ff}.\ref{ff2} since $m\geq 6$ and $k\geq 6$. 

Thus the conditions of Theorem \ref{chv} are satisfied and therefore $H$ has a Hamiltonian cycle.

\tbf{Case 2.2} ($k=4$)

By Claim \ref{AB}.\ref{SB}, we either don't have equality and thus $$d_H(u_{\frac{n}{2}-i})\geq \delta(S, B)\geq \frac{n}{2}+2\floor{\frac{n+2}{k+2}}-\frac{2n}{k}+2\geq \frac{n}{2}-\floor{\frac{n+2}{k+2}}+1\stackrel{\eqref{i}}{\geq} \frac{n}{2}-i+1,$$
where the third inequality holds by Fact \ref{ff}.\ref{ff2} since $m\geq 3$ and $k\geq 4$, and thus we are done as in the previous case; or $\delta(S, B)= \frac{n}{2}+2\floor{\frac{n+2}{k+2}}-\frac{2n}{k}+1$ and without loss of generality, $S_1=V_1$.

When $3k\leq n\leq 4k$, we have $\delta(G)\geq \frac{n}{2}-1$ and thus we are done by Claim \ref{n/2-1}.  So for remainder of this case, suppose $n\geq 5k$ (i.e. $m\geq 5$).  Also note that 
since $k=4$, \eqref{cf} reduces to $$\delta(G) \geq \frac{n}{4}+\floor{\frac{n+2}{6}}.$$

\begin{cla}\label{42}~
\begin{enumerate}[label=\emph{(\roman*)}]
\item \label{42i} $|S_2|\geq \floor{\frac{n+2}{6}}+1$.
\item \label{42ii} $\delta(S_1, B)\geq 2\floor{\frac{n+2}{6}}+1$.
\item \label{42iii} $\delta(S_2, B)\geq \delta(G)\geq \frac{n}{4}+\floor{\frac{n+2}{6}}$
\end{enumerate}
\end{cla}

\begin{proof}
\begin{enumerate}
\item Since $|S|\geq \delta(G)+1$, we have $$|S_2|=|S|-|S_1|\geq |S|-\frac{n}{4}\geq \delta(G)+1-\frac{n}{4}=\floor{\frac{n+2}{6}}+1.$$

\item Each vertex in $S_1$ has at most $|V_{2}\setminus S_{2}|$ neighbors in $V_{2}$ and thus by \ref{42i}, at least $\frac{n}{4}+\floor{\frac{n+2}{6}}-(\frac{n}{4}-|S_{2}|)\geq \floor{\frac{n+2}{6}}+\floor{\frac{n+2}{6}}+1=2\floor{\frac{n+2}{6}}+1$ 
neighbors in $B$.

\item Since $S_1=V_1$, the vertices in $S_2$ have no neighbors in $A$ and thus all of their neighbors are in $B$. \qedhere
\end{enumerate}
\end{proof}

We are in the case where $\frac{n}{2}-i\geq |A\setminus S|+1$, so if
\begin{equation}\label{i2}
\frac{n}{2}-i\leq \frac{n}{2}-\floor{\frac{n+2}{6}}-1, 
\end{equation}
then by Claim \ref{42}.\ref{42ii} we have
$$d_H(u_{\frac{n}{2}-i})\geq \delta(S_1, B)\geq 2\floor{\frac{n+2}{6}}+1\geq \frac{n}{2}-\floor{\frac{n+2}{6}}\stackrel{\eqref{i2}}{\geq} \frac{n}{2}-i+1,$$
where the third inequality holds by Fact \ref{ff}.\ref{ff2} (in particular $3\floor{\frac{n+2}{6}}+1\geq 3(\frac{n+2-4}{6})+1=\frac{n}{2}$).
Otherwise together with \eqref{i}, we have
$\frac{n}{2}-i=\frac{n}{2}-\floor{\frac{n+2}{6}},$
so by Claim \ref{42}.\ref{42i},\ref{42iii} we have
$$d_H(u_{\frac{n}{2}-i})\geq \delta(S_2, B)\geq\frac{n}{4}+\floor{\frac{n+2}{6}}\geq \frac{n}{2}-\floor{\frac{n+2}{6}}+1= \frac{n}{2}-i+1,$$
where the third inequality holds by Fact \ref{ff}.\ref{ff1} since $m\geq 5$ and $k=4$.

This completes the proof of Theorem \ref{main} and Proposition \ref{prop:n=2k}. \qed

\newpage

\section{Appendix: Numerical lemmas}\label{sec:appendix}

The level of precision in \eqref{cf} necessitates a careful handling of floors and ceilings throughout the paper.  We collect a number of such required facts here.  

First, to see how \eqref{cf} compares to \eqref{cfgjl} we note the following fact which in particular implies that \eqref{cf} is not obtained by simply replacing the strict inequality in \eqref{cfgjl} with a weak inequality (or the combination of a weak inequality and a floor).

\begin{fact}
Let $n\geq k\geq 2$ such that $n$ is divisible by $k$.  Then
\begin{equation}\label{eqc}
\dk=\ceiling{\frac{n}{2}+\frac{n}{2\ceiling{\frac{k+1}{2}}}-\frac{n}{k}} 
\end{equation}
if and only if 
\begin{enumerate}
\item $k$ is even and $n\equiv k \bmod (k+2)$
\item $k$ is odd, $n$ is even, and $n\equiv k-1 \bmod (k+1)$
\item $k$ is odd, $n$ is odd, and $n\equiv j \bmod (k+1)$ for some $j\in \{k-1, 1,3, \dots, \frac{k+1}{2}\}$
\end{enumerate}
and 
\begin{equation}\label{eqf}
\dk=\floor{\frac{n}{2}+\frac{n}{2\ceiling{\frac{k+1}{2}}}-\frac{n}{k}} 
\end{equation}
otherwise.
\end{fact} 

\begin{proof} 
\begin{enumerate}
\item 
Since $k$ divides $n$ and $k$ is even, this implies $n$ is even.  So \eqref{eqc} holds when $\floor{\frac{n+2}{k+2}}=\ceiling{\frac{n}{k+2}}$ and \eqref{eqf} holds when $\floor{\frac{n+2}{k+2}}=\floor{\frac{n}{k+2}}$.

\item 
Suppose $k$ is odd and $n$ is even.  So \eqref{eqc} holds when $\floor{\frac{n+2}{k+1}}=\ceiling{\frac{n}{k+1}}$ and \eqref{eqf} holds when $\floor{\frac{n+2}{k+1}}=\floor{\frac{n}{k+1}}$.

\item Suppose $k$ is odd and $n$ is odd.  So \eqref{eqc} holds when $\frac{n+1}{2}+\floor{\frac{n+2}{k+1}}=\ceiling{\frac{n}{2}+\frac{n}{k+1}}$ which is equivalent to $\floor{\frac{n+2}{k+1}}=\ceiling{\frac{n-\frac{k+1}{2}}{k+1}}$ and \eqref{eqf} holds when $\frac{n+1}{2}+\floor{\frac{n+2}{k+1}}=\floor{\frac{n}{2}+\frac{n}{k+1}}$ which is equivalent to $\floor{\frac{n+2}{k+1}}=\floor{\frac{n-\frac{k+1}{2}}{k+1}}$. \qedhere
\end{enumerate}
\end{proof}


In the case when $k$ is odd, the following fact shows that in order to estimate $\dk$ from below, it suffices to consider the case when $n$ is even.

\begin{fact}\label{kodd}
Let $k$ be an odd integer with $k\geq 3$ and let $n$ be divisible by $k$.  Then $$\ceiling{\frac{n}{2}}+\floor{\frac{n+2}{k+1}}-\frac{n}{k}\geq \begin{cases} \frac{n}{2}+\frac{n}{k+1}-\frac{n}{k}+\frac{1}{2}-\frac{k-2}{k+1} & : n \text{ odd }   \\
\frac{n}{2}+\frac{n}{k+1}-\frac{n}{k}-\frac{k-3}{k+1} & : n \text{ even }\end{cases}
\geq \frac{n}{2}+\frac{n}{k+1}-\frac{n}{k}-\frac{k-3}{k+1}.$$
\end{fact}

\begin{proof}
When $n$ is odd we have 
\begin{align*}
\ceiling{\frac{n}{2}}+\floor{\frac{n+2}{k+1}}-\frac{n}{k}=\frac{n+1}{2}+\floor{\frac{n+2}{k+1}}-\frac{n}{k}&\geq \frac{n+1}{2}+\frac{n+2-k}{k+1}-\frac{n}{k}\\
&=\frac{n}{2}+\frac{n}{k+1}-\frac{n}{k}+\frac{1}{2}-\frac{k-2}{k+1}
\end{align*}
and when $n$ is even we have 
\begin{align*}
\ceiling{\frac{n}{2}}+\floor{\frac{n+2}{k+1}}-\frac{n}{k}=\frac{n}{2}+\floor{\frac{n+2}{k+1}}-\frac{n}{k}&\geq \frac{n}{2}+\frac{n+2-(k-1)}{k+1}-\frac{n}{k}\\
&=\frac{n}{2}+\frac{n}{k+1}-\frac{n}{k}-\frac{k-3}{k+1}. \qedhere
\end{align*}
\end{proof}

The following two technical facts will be used throughout the proof.  

\begin{fact}\label{ff} Let $m$ and $k$ be positive integers and let $n=mk$ (note that if $k$ is even, then $n$ is even).
\begin{enumerate}[label=\emph{(\roman*)}]
\item \label{ff1} If $k$ is even, then $2\floor{\frac{n+2}{k+2}}-\frac{n}{k}\geq 2\left(\frac{n+2-k}{k+2}\right)-\frac{n}{k}=\frac{(m-2)(k-2)}{k+2}$.
\item \label{ff2} If $k$ is even, then $3\floor{\frac{n+2}{k+2}}-\frac{2n}{k}\geq 3\left(\frac{n+2-k}{k+2}\right)-\frac{2n}{k} =\frac{(m-3)(k-4)-6}{k+2}$.
\item \label{ff3} If $k$ is odd and $n$ is even, then $2\floor{\frac{n+2}{k+1}}-\frac{n}{k}\geq 2\left(\frac{n+2-(k-1)}{k+1}\right)-\frac{n}{k}= \frac{(m-2)(k-1)+4}{k+1}$.
\item \label{ff4} If $k$ is odd and $n$ is even, then $3\floor{\frac{n+2}{k+1}}-\frac{2n}{k}\geq 3\left(\frac{n+2-(k-1)}{k+1}\right)-\frac{2n}{k} =\frac{(m-3)(k-2)+3}{k+1}$.
\end{enumerate}
\end{fact}


\begin{fact}\label{ineq} Let $m$ and $k$ be positive integers with $k\geq 3$ and let $n=mk$.
\begin{enumerate}[label=\emph{(\roman*)}]
\item \label{f1} If $k$ is even and $n\geq 3k$, then $2\left(\dk\right)>(1-\frac{1}{k})n$.
\item \label{f2} If $k$ is odd and $n\geq 2k$, then $2\left(\dk\right)>(1-\frac{1}{k})n$.
\item \label{f3} If $n\geq 2k$, then $3\left(\dk\right)\geq (2-\frac{1}{k})n-\floor{\frac{n-1}{2}}-2$.
\end{enumerate}
\end{fact}

\begin{proof}
\begin{enumerate}
\item Since $n$ is divisible by $k$ and $k$ is even, $n$ is even.  Since $m=\frac{n}{k} \geq 3$ and $k \geq 3$, we have by Fact \ref{ff}.\ref{ff1}
\begin{align*}
2\left(\dk\right) - (1-\frac{1}{k})n = 2\floor{\frac{n+2}{k+2}} - \frac{n}{k} &\geq \frac{(m-2)(k-2)}{k+2}>0.
\end{align*}
 
\item Since $k$ is odd, we may assume by Fact \ref{kodd} that $n$ is even.  Since $m=\frac{n}{k}\geq 2$ and $k\geq 3$, we have by Fact \ref{ff}.\ref{ff3}
\begin{align*}
2\left(\dk\right) - (1-\frac{1}{k})n \geq 2\floor{\frac{n+2}{k+1}} - \frac{n}{k} &\geq \frac{(m-2)(k-1)+4}{k+1}>0.
\end{align*}

\item As before, if $k$ is even, then $n$ is even.  Also by Fact \ref{kodd}, if $k$ is odd, we may assume that $n$ is even.  So we have 
\begin{align*}
3\left(\dk\right) - (2 - \frac{1}{k})n + \floor{\frac{n-1}{2}} + 2 &\geq  3\floor{\frac{n+2}{2\ceiling{\frac{k+1}{2}}}}-\frac{2n}{k}+1.
\end{align*}
Now if $k$ is even, then we have $k\geq 4$ and thus by Fact \ref{ff}.\ref{ff2} and $m=\frac{n}{k}\geq 2$ we have
\begin{align*} 
3\floor{\frac{n+2}{k+2}}-\frac{2n}{k}+1\geq \frac{(m-3)(k-4)-6}{k+2}+1=\frac{(m-2)(k-4)}{k+2}\geq 0.
\end{align*}
If $k$ is odd, then since $m=\frac{n}{k}\geq 2$ and $k\geq 3$, we have by Fact \ref{ff}.\ref{ff4} that
\[
3\floor{\frac{n+2}{k+1}}-\frac{2n}{k}+1\geq \frac{(m-3)(k-2)+3}{k+1}+1=\frac{(m-2)(k-2)+6}{k+2}\geq 0.
\qedhere 
\]
\end{enumerate}
\end{proof}

\end{document}